\documentclass[leqno]{article}
\usepackage{amsmath}
\usepackage[latin1]{inputenc}
\usepackage{amssymb}
\usepackage{xcolor}
\usepackage{amsthm}                                        
\usepackage{url}
\usepackage{graphicx}
\usepackage{epsfig}
\newtheorem{theorem}{\bf Theorem}

\newtheorem{corollary}[theorem]{\bf Corollary}
\newtheorem{lemma}[theorem]{\bf Lemma}

\newtheorem{proposition}[theorem]{\bf Proposition}

\newtheorem{remark}[theorem]{\bf Remark}

\numberwithin{equation}{section}
\numberwithin{theorem}{section}
\numberwithin{figure}{section}

\def\R{\mathbb{R}}

\def\R{\mathbb{R}}

\begin{document}
\renewcommand{\thefootnote}{}
\footnotetext{Research partially supported by Ministerio de Econom\'ia y Competitividad Grant No: MTM2016-80313-P,  Junta de Andaluc\'ia Grant No. A-FQM-139-UGR18 and the `Maria de Maeztu'' Excellence
Unit IMAG, reference CEX2020-001105-M, funded by
MCIN/AEI/10.13039/501100011033} 

\title{A Weierstrass type representation for translating solitons and singular minimal surfaces} 
\author{Antonio Mart\'inez and A.L. Mart\'inez-Trivi\~no}
\vspace{.1in}
\date{}
\maketitle
{
\noindent $^1$Departament of Geometry and Topology, University of Granada, E-18071 Granada, Spain\\  \\
e-mails: amartine@ugr.es, aluismartinez@ugr.es}
\begin{abstract}
In this paper we provide a Weierstrass representation formula for translating solitons and singular minimal surfaces in ${\mathbb{R}^3}$. As application we study when the euclidean Gauss map has a harmonic argument and solve a general Cauchy problem in this class of surfaces.
\end{abstract}
\vspace{0.2 cm}


\everymath={ }

\section{Introduction.}
The theory  of critical points for the weighted  area functional \begin{equation}
{\cal A}_\varphi (\Sigma) = \int _\Sigma\mathrm{e}^\varphi d\Sigma, \label{fe}
\end{equation} 
of surfaces  $\Sigma$ in a domain $\Omega_\varphi\subseteq \mathbb{R}^3$ when 
 $\varphi$ represents the restriction on $\Sigma$ of a function depending only on the last coordinate $z$ of $\Omega$, is an area of mathematical research situated at the crossroad of several disciplines: geometric analysis, partial differential equations, mathematical physics and architecture  among others. The Euler-Lagrange   equation of \eqref{fe} is given in terms of the mean curvature vector  ${\text {\bf H}}$ of $\Sigma$ as follows: 
\begin{align} 
&{\text {\bf H}} = \left(\overline{\nabla} \varphi \right)^\perp \label{fminimal}
\end{align}
where $ \overline{\nabla} $ is  the gradient operator in $\R^3$ and  $\perp$ denotes the projection to the normal bundle of  $\Sigma$. 

The equation \eqref{fminimal} also means that $\Sigma$ is a minimal surface in  the Ilmanen space $(\Omega_\varphi,{\rm I_\varphi})$,  \cite{Il}, that is, in $\Omega_\varphi$ endowed with the conformally changed metric 
\begin{align*} 
{\rm I}_\varphi:=  \mathrm{e}^\varphi\ \langle \cdot , \cdot \rangle.
\end{align*}

\begin{itemize}
\item  When $\varphi$ is just the height function, $\varphi(z)=z$, $\Sigma$ is a translating soliton, that is, a surface such that $$ t \rightarrow \Sigma + t \vec{e}_3 $$
is a mean curvature flow, i.e. the normal component of the  velocity at  each point is equal to the mean curvature at that point. Recent advances in the understanding of their local and global geometry can be found in \cite{CSS, HMW, HIMW, HIMW2,MSHS1, MSHS2, SX, W}
\item When $\varphi(z)=\log(z)$,  $z>0$, $\Sigma$ describes the shape of a ``hanging roof'', i.e. a heavy surface in a gravitational field that, according to the architect F. Oho \cite[p. 290]{O} are of importance for the construction of perfect domes.
\item When $\varphi(z)=-2\log(z)$,  $z>0$, $\Sigma$ is a minimal surface in the half space model of the hyperbolic  space.
\end{itemize}
In the study of minimal surfaces in the Euclidean space, the classical Weierstrass-Enneper representation formula has been proved to be an extremely useful tool. The purpose of  this paper is to  provide a Weierstrass type representation formula for the class ${\cal S}_k$ of either translating solitons ($k=1$) or singular minimal surfaces ($k\neq 1$), that is, for minimal surfaces in $\mathbb{R}^3$ respect to the weighted area functional ${\cal A}_{\varphi_k}$ when  $\varphi_k$ is given by 

\begin{align*}
& \varphi_k(z) =z, \quad z\in \mathbb{R}, \quad \text{ if $k=1$},\\
&  \varphi_k(z) =\frac{2}{k-1} \log (z)  \quad z>0, \quad \text{ if $k\neq1$ }
\end{align*}

We refer to \cite{BHT,D, Du, DH, Rafa1, Rafa2, Rafa3, Rafa4, MM, N} for some  progress made in the family of singular minimal surfaces.
In Section 2, we consider  a new model of the Ilmanen space  obtained from $\Omega_{\varphi_k}$ under a change of the metric ${\rm I}_{\varphi_k}$, we give the integrability conditions  of a minimal surface in this new model and  obtain the equation satisfied by the Gauss map of any  surface  in  ${\cal S}_k$. Section 3 is devoted  to derive a representation formula for this kind of surfaces. In Section 4, we use our representation first to characterize, in terms of the Gauss map,  all examples that are  invariant  under either horizontal translations or vertical rotations and second to solve the following  general Cauchy  problem for surfaces in the class ${\cal S}_k$:
\begin{quote}
Let $\beta = (\beta_1,\beta_2,\beta_3):I\rightarrow \mathbb{R}^3$ be a regular analytic  curve and let $V:I\rightarrow \mathbb{S}^2$ be an analitic vector field along $\beta$ such that $\langle\beta',V\rangle=0$, $| \Pi\circ V|<1$ and $\beta_3>0$ if $k\neq1$, where $\Pi$ denotes the stereographic projection from the south pole. Find all surfaces in ${\cal S}_k$ containing $\beta$ with unit normal in $\mathbb{R}^3$ along $\beta$ given by $V$.
\end{quote} 
This problem has been inspired by the classical Bj\"{o}rling problem for minimal surfaces in $\R^3$, proposed by E.G. Bj\"{o}rling in 1844  and solved by H.A. Schwarz in 1890. 

Finally, we apply the solution to this Cauchy problem to study the geometry of surfaces in the class ${\cal S}_k$.

\section{The normal Gauss map equation}
For any $k\in \mathbb{R}$ we consider $\mathbb{R}^3_k$ the upper half space $\mathbb{R}^3_+$ if  $k\neq 0$ and  $\mathbb{R}^3_0=\mathbb{R}^3$. If $\{x,y,w\}$ is a coordinate system of $\mathbb{R}^3_k$ we will  take on $\mathbb{R}^3_k$  the following Riemannian metric $g_k$,
$$ g_k:= \mathrm{e}^{\eta_k(w)}( dx^2 + dy^2) + dw^2 ,$$ where \begin{align}
&\eta_k(w) = \frac{2}{k} \log (w)- \log (4),  \quad \text{if $k\neq0$},\label{eta1}\\
& \eta_k(w) = - 2\, w, \quad \text{if $k=0$}. \label{eta2}
\end{align}
It is clear that $(\mathbb{R}^3_k,g_k)$ is isometric to the Ilmanen space $(\Omega_{\varphi_k},I_{\varphi_k})$ by the following coordinate transformation $\Gamma: \mathbb{R}^3_k \rightarrow \Omega_{\varphi_k}$,
\begin{equation}\label{gamma}
\Gamma(x,y,w) = (x,y,z), \quad dz= \mathrm{e}^{-\frac{1}{2}\eta_k(w)} dw =  \mathrm{e}^{-\frac{1}{2}\varphi_k(z)} dw.
\end{equation}

By a straightforward computation, in the system of coordinates $\{x,y,w\}$, the Christoffel's symbols of $(\mathbb{R}^3_k, g_k)$ satisfy 
\begin{equation}
\label{Symbols}
\Gamma_{13}^{1}=\Gamma_{31}^{1}=\Gamma_{23}^{2}=\Gamma_{32}^{2}=\frac{1}{2}\dot{\eta}_k \ \ \text{and} \ \ \Gamma_{22}^{3}=\Gamma_{11}^{3}=-\frac{1}{2}\dot{\eta}_k\mathrm{e}^{\eta_k}, 
\end{equation}
where $(\ \dot{ }\ )$ stands the derivate with respect to $w$ and the rest of Christoffel's symbols  vanish everywhere.  

\

Let $\psi:\Sigma\rightarrow(\mathbb{R}^{3}_k,g_{k})$ be a minimal immersion, consider a local conformal parameter $\zeta=u+iv$ of $\Sigma$ on an open simply connected domain ${\cal U} \subset \mathbb{C}$ such that the induced metric $ds^2_k$  writes
\begin{equation} ds^2_k := \lambda^2 (du^2 +dv^2) = \lambda^2 |d\zeta|^2\label{imetric}
\end{equation} and set, as usual, the  Wirtinger's operators  by, 
$$
\partial_{\zeta}=\frac{1}{2}\left(\partial_{u}-i\partial_{v}\right), \quad \partial_{\overline{\zeta}}=\frac{1}{2}\left(\partial_{u}+i\partial_{v}\right).
$$
It is well known that $\psi$ is minimal if and only if $\psi$ is a harmonic map, that is, the tension field $T(\psi)={\rm trace} \nabla d\psi$ vanishes identically on $\Sigma$. But from \eqref{Symbols} and \eqref{imetric},
\begin{align*}T(\psi)& = 
2\lambda^{-2} \left( 2x_{\zeta\overline{\zeta}}+\dot{\eta}_k(w_{\zeta}x_{\overline{\zeta}}+w_{\overline{\zeta}}x_{\zeta})\right., \\
&\left.2y_{\zeta\overline{\zeta}}+\dot{\eta}_k(w_{\zeta}y_{\overline{\zeta}}+w_{\overline{\zeta}}y_{\zeta}),
2w_{\zeta\overline{\zeta}}-\dot{\eta}_k\mathrm{e}^{\eta_k}(x_{\zeta}x_{\overline{\zeta}}+y_{\zeta}y_{\overline{\zeta}})\right).
\end{align*}

To sum up, we have
\begin{proposition} 
\label{sistema1}
The following statements are equivalent
\begin{itemize}
\item $\Gamma\circ \psi$ is a minimal surface in $(\Omega_{\varphi_k},{\rm I}_{\varphi_k})$,
\item $\psi$ is a harmonic map,
\item $\psi=(x,y,w)$ satisfies,
 \begin{align}
&2x_{\zeta\overline{\zeta}}+\dot{\eta}_k(w_{\zeta}x_{\overline{\zeta}}+w_{\overline{\zeta}}x_{\zeta})=0, \nonumber\\
&2y_{\zeta\overline{\zeta}}+\dot{\eta}_k(w_{\zeta}y_{\overline{\zeta}}+w_{\overline{\zeta}}y_{\zeta})=0, \label{harmonic}\\
&2w_{\zeta\overline{\zeta}}-\dot{\eta}_k\mathrm{e}^{\eta_k}(x_{\zeta}x_{\overline{\zeta}}+y_{\zeta}y_{\overline{\zeta}})=0.\nonumber
\end{align}
\end{itemize}
\end{proposition}
Now, if we consider   $f= \mathrm{e}^{\frac{1}{2}\eta_k}x_{\zeta}$, $g= \mathrm{e}^{\frac{1}{2}\eta_k}y_{\zeta}$ and $h=w_{\zeta}$, the conformality conditions write as follows 
\begin{align}
& \lambda^2 = 2\left(\vert f\vert^{2}+\vert g\vert^{2}+\vert h\vert^{2}\right)\label{con1}\\
&f^{2}+g^{2} + h^{2}=0.\label{con2}
\end{align}
and from \eqref{harmonic} we have that

\begin{equation}
\label{sistema2}
h_{\overline{\zeta}}=\frac{1}{2}\dot{\eta}_k\left(\vert f\vert^{2} +\vert g\vert^{2}\right), \quad f_{\overline{\zeta}}=-\frac{1}{2}\dot{\eta}_k\overline{f}h,  \quad
g_{\overline{\zeta}}=-\frac{1}{2}\dot{\eta}_k\overline{g}h
\end{equation}

Now, if we introduce the complex functions 
\begin{equation}
\label{complexfunctions}
F=f-i\ g \quad \text{ and } \quad G=\frac{h}{F}, 
\end{equation}
from \eqref{con1} and \eqref{con2}, we have that  $G$ is a smooth map into the Riemann sphere and if $G$ is not constant,
\begin{equation}
\label{coordiso}
h=FG \ \ , \ \ f=\frac{1}{2}F(1-G^{2}) \ \ , \ \ g=\frac{i}{2}F(1+G^{2}).
\end{equation}
Moreover, the Gauss map $N$ of $\Gamma\circ\psi$ in the Euclidean space $\mathbb{R}^3$ is given in term of $G$ as
\begin{align*}
 N:= \left(\frac{2G}{1+\vert G\vert^{2}}, \frac{1-\vert G\vert^{2}}{1+\vert G\vert^{2}} \right).
\end{align*}
Hence, $G=\Pi\circ N$, where $\Pi$ is the stereographic projection from the point $(0,0,-1)$. We are going to say that $G$ is the \textit{euclidean Gauss map} of  $\Gamma\circ\psi$ and \textit{the normal Gauss map} of $\psi$.  

\

From the equations \eqref{sistema2} and by using \eqref{con1} and \eqref{con2}, we obtain that $\psi$ is an harmonic map if and only if
\begin{align}
\label{sistema3}
&2F_{\overline{\zeta}}=\dot{\eta}_k\, \vert F\vert^{2}\vert G\vert^{2}\overline{G},\nonumber \\
&4G_{\overline{\zeta}}=\dot{\eta}_k\, \overline{F}(1-\vert G\vert^{4})\\
& w_\zeta=FG\nonumber
\end{align}

\begin{proposition}
The normal Gauss map $G$ of $\psi$ satisfies the following complex equation 
\begin{align}
\label{equ}
&G_{\zeta\overline{\zeta}}+2\frac{\vert G\vert^{2}}{1-\vert G\vert^{4}}\overline{G}G_{\zeta}G_{\overline{\zeta}}+2k \frac{\vert G_{\overline{\zeta}}\vert^{2}}{1-\vert G\vert^{4}}G=0.
\end{align}
and
$$ \frac{G_{\overline{\zeta}}}{{1-\vert G\vert^{4}}}, \ \  \frac{\overline{G}G_{\overline{\zeta}}}{{1-\vert G\vert^{4}}}, \ \ \frac{\overline{G}^2G_{\overline{\zeta}}}{{1-\vert G\vert^{4}}}$$
are smooth functions on ${\Sigma}$
\end{proposition}
\begin{proof}
From \eqref{sistema3} and by a straighforward computation we get  
\begin{equation}
\label{paso1}
G_{\zeta\overline{\zeta}}=\frac{1}{4}\ddot{\eta}_k\, w_{\zeta}\overline{F}(1-\vert G\vert^{4})-2\frac{\vert G\vert^{2}}{1-\vert G\vert^{4}}\overline{G}G_{\zeta}G_{\overline{\zeta}}.
\end{equation}
On the other hand, the second equation of \eqref{sistema3} gives 
\begin{equation}
\label{paso2}
\vert F\vert^{2}=\frac{16}{\dot{\eta}_k^{2}}\frac{\vert G_{\overline{\zeta}}\vert^{2}}{(1-\vert G\vert^{2})^{2}}.
\end{equation}
Taking into account that $\omega_{\zeta}=FG$, \eqref{equ} follows from \eqref{eta1}, \eqref{eta2}, \eqref{paso1} and   \eqref{paso2} .

But from \eqref{coordiso} and \eqref{sistema3} we have that
$$ \frac{G_{\overline{\zeta}}}{{1-\vert G\vert^{4}}}=\frac{\dot{\eta}_k}{4} \overline{F}, \quad  \frac{\overline{G}G_{\overline{\zeta}}}{{1-\vert G\vert^{4}}}=\displaystyle(\eta_k)_{\overline{\zeta}}, \quad \frac{\overline{G}^2G_{\overline{\zeta}}}{{1-\vert G\vert^{4}}}=\frac{\dot{\eta}_k}{4} (\overline{F}- 2 f)$$ 
which are smooth functions on $\Sigma$.
\end{proof}
\begin{lemma}\label{lemaf}
Let $G:{\cal U}:\rightarrow \mathbb{C}$ be  non-constant regular solution of \eqref{equ}. If  $ \Upsilon=\displaystyle \frac{G_{\overline{\zeta}}}{{1-\vert G\vert^{4}}}\not\equiv 0$, then  the zeros of $ \Upsilon$ are isolated and of finite order.
\end{lemma}
\begin{proof}
By \eqref{equ},  $\displaystyle \Upsilon= \frac{\overline{G}_{\zeta}}{{1-\vert G\vert^{4}}}$ satisfies,
\begin{align*}
 \Upsilon_{\overline{\zeta}} + 2 |\Upsilon|^2( |G|^2-k)\overline{G}=0
\end{align*}
and then it is possible to apply the study about generalized analytic functions of Vekua (cf. \cite[Chapter III]{Ve}) to obtain that $\Upsilon$ may be locally written in the form
$$  \Upsilon = \varXi {\cal H},$$ where $\varXi$ is a non vanishing ${\cal C}^2$-function and ${\cal H}$ is holomorphic. Hence, the zeros of $ \Upsilon$ must be   isolated and of finite order. 
\end{proof} 

\begin{remark}{\rm Notice that from \eqref{con1} and \eqref{coordiso}, the induced metric  $ds^2_k$ is given by $$ds^{2}_k=\vert F\vert^{2}(1+\vert G\vert^{2})^{2}\vert d\zeta\vert^{2}.$$
Hence, using  \eqref{con1}, \eqref{con2} and \eqref{sistema3},  we can get that the Gauss curvature $K_\psi$  of $ds^2_k$ is given by 
\begin{align*} K_\psi =-\frac{\dot{\eta}_k^{2}}{4}-4\frac{\vert G_{\zeta}\vert^{2}}{\vert F\vert^{2}(1+\vert G\vert^{2})^4}+\frac{\dot{\eta}_k^{2} |G|^2 k}{(1 + |G|^2)^2}.\end{align*}}
\end{remark}

\begin{remark}
{\rm Observe that from  \eqref{sistema3}, $G$ is holomorphic if and only if $|G|\equiv 1$ and, in this case, it is clear that $G$ must be constant and  $\Gamma \circ \psi(\Sigma)$ and $\psi(\Sigma)$   lie on a vertical plane in $\mathbb{R}^3$.}
\end{remark}

\section{A representation formula}
In this section we obtain a representation formula for surfaces in the class ${\cal S}_k$. As the  case $k=0$ (i.e.  of minimal surfaces in Hyperbolic space)  was studied in \cite{K}, from  now and on we will assume that $k\neq0$.
\begin{lemma}\label{lema2}
The equation \eqref{equ} gives the integrability conditions of the system \eqref{sistema3}.
\end{lemma}
\begin{proof}
Assume $G$ is a solution of \eqref{equ} on a simply-connected domain, then  $$\frac{G_{\overline{\zeta}} \overline{G}}{1-|G|^4} d\overline{\zeta} +  \frac{\overline{G}_\zeta G}{1-|G|^4} d\zeta$$ is a  closed 1-form  and the Poincare's Lemma gives the existence of   real function $\nu$ satisfying 
$$\nu_{\overline{\zeta}} = \frac{G_{\overline{\zeta}} \overline{G}}{1-|G|^4}, \quad \nu_\zeta=  \frac{\overline{G}_\zeta G}{1-|G|^4} ,$$
By taking
$$w=\mathrm{e}^{2k\nu}, \qquad  F = 2 k w \frac{  \overline{G}_\zeta }{1-|G|^4},$$
one can check that $G$, $F$ and $w$ are solutions of \eqref{sistema3} when  $\eta_k$ is given by \eqref{eta1}.
\end{proof}
\subsection*{{\sc The case $k=1$: Translating Solitons.}}
\begin{theorem}[Weierstrass type representation for translating solitons]
\label{W1}
Let $G$ be a not holomorphic solution of \eqref{equ} defined on a simply connected domain ${\cal U}\subset \mathbb{C}$. Then the map $\widetilde{\psi}:{\cal U} \rightarrow \mathbb{R}^3$ given by
\begin{equation}\label{rw1}
\widetilde{\psi}=4\,\Re\left(\int _{\zeta_0}^\zeta\frac{\overline{G}_{\zeta}(1-G^{2})}{1-\vert G\vert^{4}} d\zeta,\int_{\zeta_{0}}^{\zeta} i\, \frac{\overline{G}_{\zeta}(1+G^{2})}{1-\vert G\vert^{4}}d\zeta, 2\int_{\zeta_{0}}^{\zeta} \frac{\overline{G}_{\zeta}G}{1-\vert G\vert^{4}}d\zeta\right)\end{equation}
is a conformal  translating soliton in $\mathbb{R}^{3}$ whose  Gauss map $N$  writes as follows
\begin{align*}
N=\left(\frac{2G}{1+\vert G\vert^{2}}, \frac{1-\vert G\vert^{2}}{1+\vert G\vert^{2}} \right).
\end{align*}
Conversely, any translating soliton which is not on a vertical plane can be locally represented in this way.
\end{theorem}
\begin{proof}
From Lemma \ref{lema2}, by taking the functions
$$ \log (w)=  4 \Re \int  \frac{\overline{G}_\zeta G}{1-|G|^4} d\zeta, \quad F = 2  w \frac{  \overline{G}_\zeta }{1-|G|^4}, \quad \eta_1= \log\frac{w^2}{4},$$
we have that $ G$, $F$ and $w$ are solutions of \eqref{sistema3} on ${\cal U}$. Thus, from \eqref{coordiso},
\begin{align*}
\psi &= \left( \Re \int \mathrm{e}^{ -\frac{\eta_1}{2}} F(1- G^2) d\zeta, \  \Re \int \mathrm{e}^{ -\frac{\eta_1}{2}}i\,  F(1+ G^2) d\zeta, \ w\right) \\
&=\left( 4 \Re \int  \frac{\overline{G}_\zeta (1-G^2)}{1-|G|^4} d\zeta, \  4 \Re \int i\, \frac{\overline{G}_\zeta (1+G^2)}{1-|G|^4} d\zeta, \mathrm{e}^{  4 \Re \int  \frac{\overline{G}_\zeta G}{1-|G|^4} d\zeta}\right).
\end{align*}
is a minimal surface in $(\mathbb{R}_1^3,g_1)$, where $ \displaystyle g_1=\frac{w^2}{4}(dx^2+dy^2) + dw^2$. But then, from \eqref{gamma}, $\widetilde{\psi}=\Gamma\circ \psi:{\cal U}\rightarrow \mathbb{R}^3$ is a conformal translator soliton given as in \eqref{rw1}.

The converse follows from Section 2.
\end{proof}

\subsection*{{\sc The case $k\neq1$: Singular Minimal Surfaces.}}

\begin{theorem}[Weiertrass representation of  singular minimal surfaces]
\label{Wk}
Let $G$ be a not holomorphic solution of \eqref{equ} defined on a simply connected domain ${\cal U}\subset \mathbb{C}$. Then the map $\widetilde{\psi}:{\cal U} \rightarrow \mathbb{R}^3$ given by
\begin{align}
\widetilde{\psi}=& \bigg( \displaystyle 4 k\Re\int_{\zeta_{0}}^{\zeta}  \frac{\overline{G}_{\zeta}(1-G^{2})}{1-\vert G\vert^{4}}\Gamma\, d\zeta, 
\displaystyle 4 k\Re\int_{\zeta_{0}}^{\zeta}  i \frac{\overline{G}_{\zeta}(1+G^{2})}{1-\vert G\vert^{4}}\Gamma\, d\zeta , 
\displaystyle\frac{2k}{k-1} \Gamma\bigg), \label{rw2}
\end{align}
where 
\begin{align*}
\Gamma = \mathrm{e}^{\displaystyle4(k-1)\Re\int_{\zeta_{0}}^{\zeta}  \frac{\overline{G}_{\zeta}G}{1-\vert G\vert^{4}}\, d\zeta}
\end{align*}
is a conformal  $\displaystyle\frac{2}{k-1}$-singular minimal surface in $\mathbb{R}^{3}_+$ whose  Gauss map $N$  writes as follows
\begin{align*}
N=\left(\frac{2G}{1+\vert G\vert^{2}}, \frac{1-\vert G\vert^{2}}{1+\vert G\vert^{2}} \right).
\end{align*}
Conversely, any singular minimal surface in $\mathbb{R}^3_+$  which is not on a vertical plane can be locally represented in this way.
\end{theorem}
\begin{proof}
By using again Lemma \ref{lema2} 
 we have that the functions $ G$, $F$ and $w$ are solutions of \eqref{sistema3} on ${\cal U}$, where now
$$ \log (w)=  4 k \Re \int  \frac{\overline{G}_\zeta G}{1-|G|^4} d\zeta, \quad F = 2 k w \frac{  \overline{G}_\zeta }{1-|G|^4}, \quad \eta_k=\frac{2}{k} \log w-\log 4.$$
  Thus, as in Theorem \ref{W1}, \begin{align*}
\psi &= \left( \Re \int \mathrm{e}^{ -\frac{\eta_k}{2}} F(1- G^2) d\zeta, \  \Re \int \mathrm{e}^{ -\frac{\eta_k}{2}}i\,  F(1+ G^2) d\zeta, \ w\right) \\
&=\left( 4 k \Re \int  w^{\frac{k-1}{k}}\frac{\overline{G}_\zeta (1-G^2)}{1-|G|^4} d\zeta, \  4 k \Re \int i\, w^{\frac{k-1}{k}} \frac{\overline{G}_\zeta (1+G^2)}{1-|G|^4} d\zeta,w\right).
\end{align*}
is a minimal surface in $(\mathbb{R}_k^3,g_k)$, where $ g_k=\displaystyle \frac{ w^\frac{2}{k}}{4}(dx^2+dy^2) + dw^2$. But then, from \eqref{gamma} we get 
\begin{align*}
z=\frac{2k}{k-1}\mathrm{e}^{ 4(k-1)\Re\int  \frac{\overline{G}_{\zeta}G}{1-\vert G\vert^{4}}\, d\zeta}.
\end{align*} and 
$\widetilde{\psi}=\Gamma\circ \psi:{\cal U}\rightarrow \mathbb{R}^3$ is a conformal $\frac{2}{k-1}$-singular minimal surface which is given   as in \eqref{rw2}.

The converse is clear from Section 2.

\end{proof}

\section{Applications}
\subsection{When $\arg(G)$ is a harmonic function}
In this section we characterize minimal surfaces for the weighted area functional ${\cal A}_{\varphi_k}$ with $\arg(G)$ a harmonic function.

Let $\widetilde{\psi}: \Sigma \rightarrow \mathbb{R}^3$ be a minimal immersion for the   weighted area functional ${\cal A}_{\varphi_k}$ 
with euclidean Gauss map $G= m \mathrm{e}^{i \nu}$, $0<m<1$. Then, from \eqref{equ}, for any complex parameter $\zeta$ on a simply connected domain ${\cal U}\subset \mathbb{C}$, we have
\begin{align}
&m (1-m^4) \nu_{\zeta\overline{\zeta}} + (1 + m^4) (\nu_{\zeta} m_{\overline{\zeta}} + \nu_{\overline{\zeta}} m_\zeta) = 0\label{eqarg}\\
&(1-m^4)m_{\zeta\overline{\zeta}} - m(1+m^4-2k \, m^{2})|\nu_\zeta|^2 + 2 m (k + m^2) |m_\zeta|^2 +\label{eqmod}\\
& +  2 k\, i \,m^2(\nu_{\zeta} m_{\overline{\zeta}}- \nu_{\overline{\zeta}} m_\zeta) =0\nonumber
\end{align}
\subsection*{{\sc Case A: $G=\zeta_0 m$, $|\zeta_0|=1$.}}
In this case up to a vertical rotation in $\mathbb{R}^3$ we can assume that $\zeta_0=1$ and, from \eqref{eqmod}, $G=m$ is a real function satisfying 
$$\frac{ m_{\zeta\overline{\zeta}}}{m_\zeta} + 2 m (k + m^2) \frac{m_{\overline{\zeta}}}{(1-m^4)} =\left( \log \left( m_\zeta (1+m^2)^{\frac{k-1}{2}}(1-m^2)^{-\frac{k+1}{2}}\right)\right)_{\overline{\zeta}}=0.$$
Hence, there exists  an harmonic  function $u:{\cal U} \rightarrow \mathbb{R}$  given by 
$$ u = 2 \Re \int (1+m^2)^{\frac{k-1}{2}}(1-m^2)^{-\frac{k+1}{2}} m_\zeta  d\zeta. $$
Now  we consider $v$ the harmonic conjugate of $u$, i.e. $\zeta = u + i v$ is a conformal parameter.  After a straightforward computation,  Theorems \ref{W1} and \ref{Wk} give that, up to a translation  in $\mathbb{R}^3$,  from $k=1$, 
$$ \widetilde{\psi} (u,v) = (2 \arctan(\tanh(u)), 2 v, \log (\cosh^2(u) + \sinh^2(u)))$$
and for $k\neq 1$, 
$$ \widetilde{\psi} (u,v) = \left(2k \mathcal{I}(u), 2 k v, \frac{2k}{k-1}\left(\frac{1+m(u)^{2}}{1-m(u)^{2}}\right)^{\frac{k-1}{2}}\right), $$
where $\mathcal{I}$ is a smooth function depending only of $u$ such that
$$\mathcal{I}'(u)=\frac{m'(u)}{1+m^{2}(u)}\left(\frac{1+m(u)^{2}}{1-m(u)^{2}}\right)^{\frac{k-1}{2}}$$
and $m$ is a solution of the following ODE,
$$ m'  ( 1+ m^2)^{\frac{k-1}{2}}  =(1-m^2)^ {\frac{k+1}{2}},$$
and $(\ '\ )$ means derivative respect $u$.

\

The above expressions give the following 
\begin{theorem} If $G=\zeta_0 \, m$, $|\zeta_0|=1$, is not constant, then either $k=1$ and $\widetilde{\psi}(\Sigma)$ lies on a grim reaper cylinder or $k\neq1$ and $\widetilde{\psi}(\Sigma)$  lies on a $\frac{2}{k-1}$-singular minimal catenary cylinder.
\end{theorem}
\subsection*{{\sc Case B:} $G=m\,\mathrm{e}^{i \nu}$, $\nu$ {\rm not constant and} $ \nu_{\zeta\overline{\zeta}} =0$.}
In this case on ${\cal U}$ we can take the following  conformal parameter $ \zeta = u + i \nu$ for some harmonic function $u$ and from \eqref{eqarg} $m_\nu\equiv 0 $, that is, $m$ is a function depending only on $u$ which satisfies, from \eqref{eqmod},  the following ODE:
$$
\frac{m - m'}{\sqrt{1-m^4}} = \left(\frac{1-m^2}{1+m^2}\right)^{\displaystyle\frac{k}{2}} \mathrm{e}^\phi, \quad \phi(u)= \int \frac{m^4+1-2 k m^2}{1-m^4} du$$
Now, as in the Case A, by using the above expressions and Theorems \ref{W1} and \ref{Wk}, we have that, up to a  translation in $\mathbb{R}^3$, either $k=1$ and 
\begin{align*}
\widetilde{\psi} &=\left(4\frac{m-m'}{1-m^2} \cos(v),\ 4\frac{m-m'}{1-m^2} \sin(v), \ \log\frac{1+m^2}{1-m^2} - 4 \int\frac{m^2}{1-m^4} du\right)
\end{align*}
or $k\neq 1$ and 
\begin{align*}
\widetilde{\psi} &=\left(4k \gamma \frac{m-m'}{1-m^2} \cos(v),\ 4k \gamma\frac{m-m'}{1-m^2} \sin(v), \frac{2 k}{k-1} \gamma\right), 
\end{align*}
where
\begin{align*}
\gamma &= \left(\frac{1+m^2}{1-m^2}\right)^{\frac{k-1}{4}}\mathrm{e}^{{(1-k)}\int\frac{m^2}{1-m^4}du},
\end{align*}
which gives
\begin{theorem} If $G=m\mathrm{e}^{i \nu}$ with $\nu$ not constant and $ \nu_{\zeta\overline{\zeta}} =0$, then $\widetilde{\psi}(\Sigma)$ lies on either  a vertical revolution translating soliton when $k=1$ or a vertical revolution $\frac{2}{k-1}$-singular minimal surface when $k\neq 1$.
\end{theorem}

\subsection{The Cauchy's problem.}
Now we deal with   the general Cauchy Problem for  surfaces in the class ${\cal S}_k$ specified in Section one.  To solve it we shall called any pair $\beta$, $V$ in the conditions of that problem {\sl a pair of Bj\"orling data}.
\begin{theorem}\label{cp}
For any $k\in \mathbb{R}$, $k\neq 0$, there exists a unique surface in the class ${\cal S}_k$ which is a  solution to the Cauchy problem with  Bj\"orling data $\beta=(\beta_1,\beta_2,\beta_3)$, $V=(V_1,V_2,V_3)$. This solution, $$\psi:{\cal U}=I\times]-\epsilon,\epsilon[\subseteq\mathbb{C}\longrightarrow \mathbb{R}^3,$$ can be constructed in a neigbourhood of $\beta$ as follows: let $G:{\cal U}\rightarrow \mathbb{C}$ be the unique solution to the following system of Cauchy-Kowalewski's type, {\rm \cite{Pe}},
\begin{equation}\label{ckp}
\left\{
\begin{array}{l}\displaystyle
G_{\zeta\overline{\zeta}}+2\,\frac{\vert G\vert^{2}G_{\zeta}G_{\overline{\zeta}}}{1-\vert G\vert^{4}}\ \overline{G}+2k\,\frac{ \vert G_{\overline{\zeta}}\vert^{2}}{1-\vert G\vert^{4}}\ G=0,\\ 
\\
\displaystyle
G(u,0) =\frac{\phi_3(u)}{\phi_1(u) - i \phi_2(u)}= -\frac{\phi_1(u) + i \phi_2(u)}{\phi_3(u)}, \\ \\
G_{\overline{\zeta}}(u,0)=\left\{\displaystyle
\frac{1-|G(u,0)|^4}{4}\left(\overline{\phi}_{1}(u)+i\overline{\phi}_{2}(u) \right), \text{ if } k=1, \atop \displaystyle
\frac{1-|G(u,0)|^4}{2(k-1)\beta_{3}}\left(\overline{\phi}_{1}(u)+i\overline{\phi}_{2}(u)\right), \text{ if } k\neq 1,
\right.
\end{array}\right.
\end{equation}
where 
\begin{equation}\label{phi}
\phi(u)=(\phi_1(u),\phi_2(u),\phi_3(u))=\frac{1}{2}(\beta'(u)-i\beta'(u)\wedge V(u)), \quad u\in I.
\end{equation}
Then $\psi$ is given, up to an appropriate translation, by \eqref{rw1} if $k=1$ and  by  \eqref{rw2} if $k\neq1$ and
\begin{align*}
N=\left(\frac{2G}{1+\vert G\vert^{2}}, \frac{1-\vert G\vert^{2}}{1+\vert G\vert^{2}} \right)
\end{align*}
is its  Gauss map.

\end{theorem}
\begin{proof}
First we will prove the uniqueness part and how the solution can be written.  Let $\beta(u)$ and $V(u)$ be the Bj\"orling data defined on a real interval $I$ and consider $\widetilde{\psi}$ a minimal surface respect to the weighted area functional ${\cal A}_{\varphi_k}$ solving the Cauchy problem for these data. Then, we can parametrize  this surface conformally on a neighbourhood of $\beta$ as $\widetilde{\psi}:{\cal U}=I\times]-\epsilon,\epsilon[\subseteq\mathbb{C}\longrightarrow \mathbb{R}^3$ so that  
 \begin{itemize}
 \item[(i)] $\widetilde{\psi}(u,0)=\beta(u)$ for all  $u\in I$, being $\zeta= u + i v$.
 \item[(ii)]  The unit normal of $\widetilde{\psi}$ along $\beta(u)$ is $V(u)$.
 \end{itemize} 
As $\widetilde{\psi}$ is a conformal immersion and $|\Pi\circ V|<1$, we observe that 
\begin{align}
&\phi(u)=\widetilde{\psi}_{\zeta}(u,0)=\frac{1}{2}(\beta'(u)-i\beta'(u)\wedge V(u)),\label{inicond}\\
&|G(u,0)|^2= \left|\frac{\phi_3(u)}{\phi_1(u) - i \phi_2(u)}\right|^2= \left|\frac{\phi_1(u) + i \phi_2(u)}{\phi_3(u)}\right|^2=\left|\frac{1-V_3}{1+V_3}\right|^2<1\label{wp}
\end{align}
where $G$ is  the euclidean Gauss map of $\widetilde{\psi}$. Thus, from  \eqref{con2}, \eqref{complexfunctions}, \eqref{sistema3}, \eqref{equ}, \eqref{inicond} and \eqref{wp} the system \eqref{ckp} is well posed and $G$ must be its unique  solution. Now, by using our Weierstrass representation, $\widetilde{\psi}$ can be recovered either as in \eqref{rw1} if $k=1$ or as in \eqref{rw2} if $k\neq 1$. To sum up, we have proved that the euclidean Gauss map $G$ of $\widetilde{\psi}$  is completely determined in a neighbourhood of $\beta$ by the Bj\"orling data $\beta$, $V$ and the solution we started with can be expressed in terms of $G$. Thus, the analyticity of surfaces in the class ${\cal S}_{k}$ gives  the uniqueness.

\

To prove the existence we start with the Bj\"orling data $\beta$, $V$ and take $\phi$ as in \eqref{phi}. Then \eqref{ckp} is a well-defined   system of Cauchy-Kowalewski's type and there exists a unique solution $G$ well defined in an open subset ${\cal U}\subseteq \mathbb{C}$ containing $I$.  Now we we distinguish two cases:
\begin{itemize}
\item If $k=1$, we may consider $\widetilde{\psi}:{\cal U}\rightarrow \mathbb{R}^3$ as in \eqref{rw1} which verifies, 
\begin{align*}
&<\widetilde{\psi}_{\zeta},\widetilde{\psi}_\zeta> =0,\\
& |\widetilde{\psi}_\zeta|^2 = 8|\Upsilon|^2 (1 + |G|^2)^2, \quad \text{with } \Upsilon= \frac{G_{\overline{\zeta}}}{{1-\vert G\vert^{4}}},\\
& \widetilde{\psi}_{\zeta}(u,0) = \phi(u) = \beta'(u) - i \, \beta'(u)\wedge V(u).
\end{align*}
Then,  from Theorem \ref{W1} and as $|\widetilde{\psi}_{\zeta}(u,0)|^2= 2|\beta'(u)|^2>0$, we have that up to an appropriate translation ${\cal U}$ can be choose  so that  $\widetilde{\psi}$  is a translating soliton solving  the Cauchy problem.
\item If $k\neq 0, 1$, then we consider $ \widetilde{\psi}_{\zeta}:{\cal U} \rightarrow \mathbb{R}^3$  as in \eqref{rw2}. In this case and by a straightforward computation we have
\begin{align*}
&<\widetilde{\psi}_{\zeta},\widetilde{\psi}_\zeta> =0,\\
& |\widetilde{\psi}_\zeta|^2 =8 k^2\, \Gamma^2 |\Upsilon|^2 (1 + |G|^2)^2, \quad \text{with } \Upsilon= \frac{G_{\overline{\zeta}}}{{1-\vert G\vert^{4}}},\\
& \widetilde{\psi}_{\zeta}(u,0) = \frac{\widetilde{\psi}_3(u,0)}{\beta_3(u)}\phi(u) = \frac{\widetilde{\psi}_3(u,0)}{\beta_3(u)}(\beta'(u) - i \, \beta'(u)\wedge V(u)).
\end{align*}
Hence, there exists a positive constant $c_0>0$, such that $c_0 \widetilde{\psi}_3(u,0)=\beta_3(u)$ and   $ |\widetilde{\psi}_{\zeta}(u,0)|^2=2c_0^2|\beta'(u)|^2>0 $. Thus,  from  Theorem \ref{Wk}, ${\cal U}$ can be choose so that, the solution to the Cauchy problem in the class ${\cal S}_k$ is  $c_0 \widetilde{\psi}\circ {\cal T}$ for an appropriate horizontal translation ${\cal T}$.
\end{itemize}
\end{proof}
\begin{remark}
In the  case $k=0$ which  corresponds to the class of minimal surfaces in the hyperbolic space, the solution to the Cauchy problem can be obtained by using the Weierstrass representation given in {\rm \cite{K}} and following a similar reasoning as in Theorem \ref{cp}.
\end{remark}
Arguing as in \cite[Theorem 12, Corollary 13]{GM} we also can prove the following geometric consequences of Theorem \ref{cp}: 
\begin{corollary}{{\rm (Generalized Symmetry principle)}} Let $\Phi$ be a positive rigid motion in $\mathbb{R}^3$ that leaves invariant the class of surfaces ${\cal S}_k$. If $\Phi$ is a symmetry in the Bj\"orling data of the Cauchy problem for ${\cal S}_k$, then $\Phi$ induces a symmetry in the resulting surface. 
\end{corollary}
\begin{corollary}{{\rm(Periodic Bj\"orling data)}} When the the Bj\"orling data $\beta(u)$ and $V(u)$ are $T$-periodic, the surface in ${\cal S}_k$ solving the Cauchy problem has the topology of a cylinder near $\beta$ and any surface in ${\cal S}_k$ with the topology of a cylinder is recovered in this way.
\end{corollary}

\end{document}